\documentclass[11pt]{amsart}
\usepackage{amsfonts,amssymb,latexsym,amsmath, amsxtra}
\usepackage[all]{xy}
\usepackage[dvips]{graphics}

\usepackage[OT2,T1]{fontenc}
\DeclareSymbolFont{cyrletters}{OT2}{wncyr}{m}{n}
\DeclareMathSymbol{\Sha}{\mathalpha}{cyrletters}{"58}

\theoremstyle{plain}
\newtheorem{theorem}{Theorem}[section]
\newtheorem{corollary}[theorem]{Corollary}

\newtheorem{lemma}[theorem]{Lemma}
\newtheorem{proposition}[theorem]{Proposition}
\newtheorem*{conjecture*}{Conjecture}

\theoremstyle{definition}

\theoremstyle{remark}
\newtheorem*{remark}{Remark}
\newtheorem*{remark*}{Remark}

\numberwithin{equation}{section}

\newcommand{\N}{\mathbb N}
\newcommand{\Z}{\mathbb Z}
\newcommand{\C}{\mathbb C}
\newcommand{\T}{\mathbb T}
\newcommand{\cC}{\mathcal{C}}

\newcommand{\cM}{\mathcal{M}}

\def\H{\mathbb H}

\def\calM{\mathcal{M}}

\def\SL{\rm SL}

\def\vt{\vartheta}

\def\({\left(}
\def\){\right)}

\newcommand{\re}[1]{\text{Re}\(#1\)}
\newcommand{\im}[1]{\text{Im}\(#1\)}

\newcommand{\abs}[1]{\left|#1\right|}

\newcommand{\crank}{\text{crank}}

\def\k2{\frac{k}{2}}

\begin{document}

\title[The Crank Statistic for Partitions]
{Families of Quasimodular Forms and Jacobi Forms: \\
The Crank Statistic for Partitions}

\author{Robert C. Rhoades}
\address{Department of Mathematics \\Stanford University \\ Stanford, CA 94305 \\ U.S.A. }
\email{rhoades@math.stanford.edu}

\thanks {The author is supported by
an NSF Mathematical Sciences Postdoctoral Fellowship.}

\date{\today}


\begin{abstract}
Families of quasimodular forms arise naturally in many situations such as curve counting on Abelian surfaces
and counting ramified covers of orbifolds. In many cases the family of
quasimodular forms naturally arises as  the coefficients  of a Taylor expansion of a Jacobi form. 
In this note we give examples
of such expansions that arise in the study of partition statistics. 

The crank partition statistic has gathered much interest recently.   For instance, Atkin and Garvan
showed that the generating functions for the moments of the crank statistic are quasimodular forms. 
The two variable generating function for the crank partition statistic is a Jacobi form. 
Exploiting the structure inherent in the  
Jacobi theta function we construct explicit expressions for the functions of Atkin and Garvan.   
Furthermore, this perspective opens the door for further investigation including a study of the moments in 
arithmetic progressions.  We conduct a thorough study of the crank statistic restricted to a residue class modulo 2.   
\end{abstract}

\maketitle


\section{Introduction}
A well known and useful fact is that the graded ring of holomorphic modular forms for $\SL_2(\Z)$, 
denoted $M_*(\SL_2(\Z))$,
is generated by the weight 4 and weight 6 Eisenstein series
$E_4(\tau)$ and $E_6(\tau)$, $\tau\in \H = \{z \in \C: \im{z}>0\}$.  
The weight $k$ Eisenstein series on $\SL_2(\Z)$, for $k\ge2$ even, is given by 
\begin{equation}
E_k(\tau) := 1-\frac{2k}{B_k}  \sum_{n=1}^\infty\sigma_{k-1}(n)q^n,
\end{equation}
with  $q :=e^{2\pi i \tau}$, $\sigma_{k-1}(n):= \sum_{d\mid n} d^{k-1}$, and $B_k$ the $k$th Bernoulli number.

On the other hand, $E_2(\tau)$ is 
not a modular form, but a quasimodular form.
Quasimodular forms are elements of the smallest algebra which contains the classical modular forms and 
which is closed under differentiation $\frac{1}{2\pi i} \frac{d}{d\tau} = \frac{1}{q} \frac{d}{dq}$. 

Families of quasimodular forms arise naturally throughout mathematics.  The following are three such examples.
\begin{enumerate}
\item Dijkgraaf \cite{D} and Kaneko and Zagier \cite{KZ}
 proved that for each $g>1$ the generating function for the 
 number of genus $g$ and degree $d$ covers of an elliptic curve with prescribed ramification is a quasimodular form.
The family of generating functions indexed by the genus $g$ is a family of quasimodular forms. 

\item For a lattice $L\subset \C$ let $\T^2 = \C/L$. 
Eskin and Okounkov proved that there is a two-parameter family of 
 quasimodular forms counting ramified covers of the pillowcase orbifold $\T^2/\pm 1$ formed by taking the quotient 
 by the automorphism $z\mapsto -z$. 
 The family of quasimodular forms is indexed by 
 a partition $\mu$ and  a partition of an even number into odd parts $\nu$.  The 
 generating function for the number of degree $d$ covers with ramification data determined by $\mu$ and $\nu$,
  $Z(\mu, \nu;q)$, is a  quasimodular form.
   
\item Andrews and Rose \cite{AR} proved there is a one parameter family of quasimodular forms 
 arising from a curve counting problem on Abelian surfaces.  They show that for each genus $g$,
 the generating function for the number of hyperplane sections which are hyperelliptic curves of genus $g$
 with $\delta$ nodes on a generic polarized abelian surface is a quasimodular form. 
\end{enumerate}

For each of these families there is a Jacobi form that may be viewed as the generating function for 
the family of quasimodular forms.  Precisely, each of these families arise as the coefficients of a Taylor 
expansion of a Jacobi form.  

Jacobi forms may be understood as two variable automorphic forms satisfying an elliptic 
transformation and a modular transformation. 
Eichler and Zagier (Theorem 3.1 of  \cite{EZ85}) 
have shown 
that a suitable ``correction'' to each Taylor coefficient with respect to the elliptic variable (see Section \ref{sec:Crank})
of a Jacobi form is a modular form.
In Eskin and Okounkov's work \cite{EO} and in our context (see Section \ref{sec:Crank}) 
the Taylor expansion with respect to the elliptic variable
yields arithmetically interesting generating functions.  
In both of these cases, the
uncorrected Taylor coefficients are of interest.  The uncorrected coefficients are, generally, 
quasimodular forms of mixed weight rather than modular forms of a fixed weight.

In this note the crank statistic for partitions is presented as an example of tis general phenomenon. 
Our focus is to produce explicit expressions for the generating functions of the moments of these statistics and 
the moments restricted to arithmetic progressions.  
As a consequence we deduce congruences for the 
coefficients of the moment generating functions. 
Furthermore, we exploit the structure of the full Jacobi form to give asymptotics for the moments themselves. 

%



\section{Moments of the Crank Paritition Statistic}\label{sec:Crank}
Dyson \cite{Dy44} conjectured the existence of a statistic, the ``crank'',
that would provide a combinatorial explanation of 
Ramanujan's congruences for the partition function modulo 5, 7, and 11. 
Garvan \cite{Gar88} found the crank
statistic, and together with Andrews \cite{AndG88}, presented the following
definition.  Let $o(\lambda)$ denote the number of ones in
$\lambda$,
and  $\mu(\lambda)$
denote the number of parts strictly larger than $o(\lambda).$  The crank of $\lambda$ is defined by 
\begin{equation}\label{definecrank}
\crank(\lambda) := \begin{cases} \text{largest part of} \; \lambda \qquad &
\text{if} \; o(\lambda) = 0 \\
\mu(\lambda) - o(\lambda) & \text{if} \; o(\lambda) > 0
\end{cases}.
\end{equation}
Let $\calM(m,n)$ be the number of partitions of $n$
with crank $m$. The two-parameter
generating function may be written as  \cite{AndG88, AS54}
\begin{align*}
C(x;q) & := \sum_{\substack{m \in \Z \\ n \geq 0}} \calM(m,n) x^m q^n =
\prod_{n \geq 1} \frac{1-q^n}{(1-xq^n)(1-x^{-1}q^n)}.
\notag
\end{align*}

The study of the crank, and other partition statistics, has led to a better understanding of the partition function.  
For instance, Mahlburg \cite{M} showed that when $x$ is specialized to be a root of unity, $C(x;q)$ is,
 up to a power of $q$,
a modular form.  
He then used the theory of modular forms to establish 
congruences for the crank statistic,  resulting in infinitely many combinatorial congruences for the 
partition function.  

Atkin and Garvan \cite{AG03} studied the distribution of the crank statistic by considering its moments. 
For a nonnegative integer $k$, define the $k$-th  crank moment as
\begin{align}\label{defineRC}
M_k(n):=\sum_{m \in \Z} m^k\, \calM(m,n).
\end{align}
They showed that for odd $k$ $M_k(n) =0$ and  
for each $\ell \in \N$ the generating function
\begin{align}
C_{2\ell}(q) := \sum_{n\geq 0} M_{2\ell}(n)q^{n} 
\end{align}
is a quasimodular form. 
Atkin and Garvan construct this family of quasimodular forms recursively. 
  
We show that this family of generating functions naturally appears as the Taylor coefficients 
of a Jacobi form.  
Let $x := e^{2\pi i u}$ and $q := e^{2\pi i \tau}$. Abusing notation, 
$q^{-\frac{1}{24}} \sin(\pi u)^{-1} C(u; \tau) = q^{-\frac{1}{24}} \sin(\pi u)^{-1} C(x;q)$ is 
a weight $\frac{1}{2}$ index $-\frac{1}{2}$ Jacobi form
on $\SL_2(\Z)$ with multiplier.
With $Z$ a complex number let
\begin{equation}
\cC(Z;q) := \sum_{k=0}^\infty C_k(q) \frac{Z^k}{k!}
\end{equation}
be the exponential generating functions for the crank moment
generating functions.
Rearranging the order of summation we see that 
$$\cC(Z;q) = \sum_{ n\ge 0} \sum_{m\in \Z} \cM(m,n) q^n \sum_{k\ge 0} 
\frac{(Zm)^k}{k!} = \sum_{n\ge 0} \sum_{m\in \Z} \cM(m,n) (e^Z)^m q^n    = C\(e^{Z}; q\).$$  
Hence the Taylor expansion with respect to $u$ 
of the Jacobi form $C(u;\tau)$ encodes the crank moments.  

In this note, we focus on concrete instances and describe the Taylor coefficients, and hence moment generating 
functions, explicitly. 
Define 
\begin{equation}
\Phi_{k-1}(\tau) :=  \frac{B_k}{2k} \(1-E_k(\tau)\) = \sum_{n=1}^\infty \sigma_{k-1}(n)q^n.
\end{equation}
\begin{theorem}\label{thm:ag4.8}For $\ell \ge 1$ we have 
$$C_{2\ell }(q) = \frac{1}{(q)_\infty} \sum_{1\le k \le \ell} \sum_{
\begin{subarray}{c} i_1 + \cdots + i_k = \ell \\ i_1, i_2, \cdots, i_k >0 \end{subarray}} 
 \frac{2^k(2\ell)!}{k! (2i_1)!\cdots (2i_k)!} \Phi_{2i_1-1}(\tau) \Phi_{2i_2-1}(\tau)\cdots \Phi_{2i_k-1}(\tau),$$
 with $(q)_\infty := (q;q)_\infty$ and $(x;q)_\infty := \prod_{n\ge 0} (1-xq^{n})$.
 Hence $(q)_\infty C_{2\ell}(q)$ is a mixed weight quasimodular form on $\SL_2(\Z)$
with highest weight $2\ell$. 
\end{theorem}

\begin{remark} 
Atkin and Garvan \cite{AG03} used a recurrence relation to deduce Theorem \ref{thm:ag4.8}.  However, 
the constants appearing in Theorem \ref{thm:ag4.8} are not given. 
\end{remark}

\section{Moments in Arithmetic Progressions}

This perspective opens the door for further exploration of the distribution of the crank. For example, 
we may investigate the crank moments in arithmetic progressions, such as 
$$\sum_{m\equiv A\pmod{B}} 
m^k \cM(m,n)$$
for any nonnegative integers $A$ and $B$.  In this section we give some results in the case $B=2$. 

The case $k=0$ and arbitrary $B$ was considered by Mahlburg \cite{M}.  The cases for $k=0$ and 
$B = 2, 3, 4$ have been studied explicitly by a number of authors, namely Andrews and Lewis \cite{AL}
and Choi, Kang, and Lovejoy \cite{CKL}. 
To understand these moments it is enough to have an explicit description of the twisted crank moments
\begin{align}
M_k(\zeta,n) := \sum_{m\in \Z} \zeta^m m^k \cM(m,n),
\end{align}
with $\zeta$ a root of unity.
Define the twisted crank moment generating function by 
\begin{align}
C_{k}(\zeta,q) := \sum_{n\geq 0} M_{k}(\zeta,n)q^{n}.
\end{align}
For concreteness, we consider the case when $\zeta =-1$ and give the 
 following explicit description of these moment generating functions.
\begin{theorem}\label{thm:twist}
For $k\ge 1$ define 
\begin{equation}\label{eqn:Fdef}
F_{2k}(\tau): = 2^{2k} \Phi_{2k-1}(2\tau) - \Phi_{2k-1}(\tau).
\end{equation}
For $\ell \ge 1$ we have 
$$C_{2\ell}(-1,q) = (q)_\infty(-q)_\infty^2 \sum_{1\le k \le \ell} \ \sum_{\begin{subarray}{c}
i_1 + \cdots + i_k = \ell \\ i_1, i_2, \cdots, i_k >0 \end{subarray}} 
 \frac{2^k(2\ell)!}{k! (2i_1)!\cdots (2i_k)!} F_{2i_1}(\tau) F_{2i_2}(\tau)
 \cdots F_{2i_k}(\tau).$$
Hence $\frac{ C_{2\ell}(-1,q)}{(q)_\infty(-q)_\infty^2}$ is a mixed weight quasimodular form on $\Gamma_0(2)$
with highest weight $2\ell$. 
\end{theorem}

We give two applications of this perspective. First, 
since quasimodular forms are $p$-adic modular forms in the sense of Serre, for instance 
$E_2(\tau) \equiv E_{p+1}(\tau) \pmod{p}$,
it is routine to deduce infinitely many congruences for the moments (for instance, see \cite{tr}). 
\begin{corollary}
For every $\ell$ there exists infinitely many 
primes $p$ and infinitely many non-nested arithmetic progressions $\{An+B:n\in \N\}$
such that 
$$M_{2\ell}(-1,An+B) \equiv 0 \pmod{p}$$
for all $n\in \N$. 
\end{corollary}

Second, we give precise asymptotics for the twisted 
crank moments $M_{2\ell}(-1,n)$ as $n\to \infty$.  
The asymptotics of $C_{2\ell}(\tau)$ as $\tau$ approaches 
rational numbers yields asymptotic information about $M_{2\ell}(-1,n)$.  
Rather than computing the asymptotics after Taylor expanding the Jacobi form
we compute the asymptotics then Taylor expand. This allows a unified approach toward understanding 
the asymptotics of $M_{2\ell}(-1,n)$ which would otherwise be out of reach. 
This idea was first used 
by the author with Bringmann and Mahlburg
\cite{BMR} to produce precise asymptotics for the moments of the crank and rank 
partition statistics.


Define  
\begin{equation}\label{eqn:alpha}
\alpha(a,b,c):= \frac{(2(a+b+c))! (-1)^{b+c}}{(2a)!(2b)!c! 4^{a+b+c}} E_{2b}
\end{equation}
where $E_{2b}$ is the Euler number given by $\sum_n E_n \frac{x^n}{n!} = \frac{1}{\cosh(x)}$
and for $k\ge 1$ define the Kloosterman sum 
\begin{equation}\label{Kloosterman}
A_{k}(n):= \sqrt{\frac{k}{24}} \sum_{\begin{subarray}{c} x\pmod{24k} \\ x^2 \equiv -24n+1 \pmod{24k}\end{subarray}}
 \(\frac{12}{x}\) e\(\frac{x}{12k}\)
\end{equation}
where the sum runs over residue classes modulo $24k$. 

\begin{theorem}\label{thm:crankAsymptotic} 
As $n\to \infty$ we have 
\begin{align*}
M_{2\ell}(-1,n) = \ &\pi \sum_{1\le k \leq
\sqrt{n}/2 }\frac{(-1)^{k+\lfloor \frac{k+1}{2}\rfloor} A_{2k}\(
n-\frac{k(1+(-1)^k)}{4}\) }{k}\sum_{a+b+c=\ell}(2k)^c
\alpha(a, b, c) \\ & \hspace{1in} \times (24n-1)^{b+\frac{c}{2}-\frac{1}{4}}
I_{\frac{1}{2}-2b-c}\left(\frac{\pi
\sqrt{24n-1}}{12k}\right) +O\Big(n^{2\ell+ \epsilon}\Big)
\end{align*}
where $I_\nu$ is the modified Bessel function of order $\nu$. 
\end{theorem}

This theorem and the fact that as $y\to\infty$ we have 
$$I_\nu(y) \sim \frac{1}{\sqrt{2\pi y}} e^{y}$$
readily gives the leading order asymptotics for the twisted crank moment. 
\begin{corollary}
In the notation above, 
$$M_{2\ell}(-1,n)\sim  (-1)^{n}\abs{E_{2\ell}} 2^{\ell-1} 3^\ell  n^{\ell-\frac{1}{2}} e^{\pi \sqrt{\frac{n}{6}}}.$$
\end{corollary}

This should be compared with  
$$M_{2\ell} (n)  \sim 2^{3\ell-2} 3^{\ell-\frac{1}{2}} (1-2^{1-2\ell})\abs{ B_{2\ell}}\cdot n^{\ell-1} 
e^{\pi \sqrt{\frac{2n}{3}}},$$
 which was first obtained by Bringmann, Mahlburg, and the author in \cite{BMR1}.  
Hence, the twisted moment is of exponentially smaller order of magnitude.  
Also, we see that as $n\to\infty$ the sign of $M_{2\ell}(-1,n)$ depends on the parity of $n$, 
which suggests that this is true for all $n$.  Andrews and Lewis \cite{AL} proved this is the case for $\ell =0$, namely 
$(-1)^n M_0(-1,n) >0$. 

Section \ref{sec:classical} 
contains some preliminary results on the Taylor expansion of the classical Jacobi theta function. 
Section \ref{sec:crank} uses these results to establish Theorems \ref{thm:ag4.8}, \ref{thm:twist}, and
 \ref{thm:crankAsymptotic}.  
The final section gives some 
discussion of the rank partition statistic, which is not covered by the method of this paper. 

\section{Jacobi's Theta Function}\label{sec:classical}
In this section we give two Taylor expansions of the Jacobi theta function.  
The first expansion is with respect to the elliptic variable and the second is with respect 
to a shift of the elliptic variable. 
Jacobi's theta function is given by 
\begin{equation}
\vt(u; \tau) := \sum_{\nu \in \Z + \frac{1}{2}} e^{\pi i\nu^2 \tau+ 2\pi i \nu u} = -2\sin(\pi u) q^{\frac{1}{8}} \prod_{n=1}^\infty 
(1-q^n) (1-x q^n) (1-x^{-1}q^n) 
\end{equation}
where the second equality is Jacobi's triple product identity.
The Dedekind eta function is
$$\eta(\tau):= \sum_{n\in \Z} \(\frac{12}{n}\)q^{n^2/24} = q^{\frac{1}{24}} \prod_{n=1}^\infty
(1-q^n).$$
\begin{proposition}\label{prop:jacobi} With $Z = 2\pi i u$ and $F_\ell$ as defined in \eqref{eqn:Fdef} for even $\ell$ and 0 
otherwise
we have 
$$\vt(u;\tau) = -2\sin(\pi u)  \eta^3(\tau) \exp\( -2 \sum_{\begin{subarray}{c} 
\ell \text{ even} \\ \ell >0 \end{subarray}} \frac{Z^\ell}{\ell!} \Phi_{\ell-1}(\tau) \)$$
and 
$$\vt\(u+\frac{1}{2};\tau\) =  -2\cos(\pi u)  \frac{\eta(2\tau)^2}{\eta(\tau)} \exp\(-2 \sum_{
\begin{subarray}{c} \ell >0\\ \ell \text{ even} \end{subarray}} \frac{Z^\ell}{\ell!} F_{\ell}(\tau) \).$$
\end{proposition}

\begin{proof}
We have \begin{align*}
\log\( -\frac{\vt(u;\tau)}{2 \sin(\pi u) \eta^3(\tau)} \) =& 
\log\(\prod_{n=1}^\infty (1-x q^n) (1-x^{-1}q^n)(1-q^n)^{-2}\) \\
=& - \sum_{n\ge 1} \sum_{r\ge 1} (x^r + x^{-r} - 2) \frac{1}{r} q^{nr} 
= -2\sum_{ \ell \text{ even}, \ell >0} \frac{Z^\ell}{\ell!} \sum_{n, r\ge 1} r^{\ell-1} q^{nr} \\
=& -2 \sum_{\ell \text{ even}, \ell>0} \frac{Z^\ell}{\ell!} \Phi_{\ell-1}(\tau)
\end{align*}

Note that 
$\vartheta(\frac{1}{2}; \tau) = -2 q^{\frac{1}{8}} \prod_{n\ge 1} (1-q^n)(1+q^n)^{2} = -2\frac{\eta(2\tau)^2}{\eta(\tau)}.$
The second result follows from a similar calculation to the previous one and 
\begin{align*}
\sum_{n\ge1} \sum_{r\ge 1} (-1)^r r^{\ell-1}q^{nr} =& 2\sum_{n\ge 1} 
\sum_{r\ge 1, r \text{ even}} r^{\ell-1} q^{nr} - \sum_{n\ge 1} \sum_{r\ge 1}r^{\ell-1} q^{nr} \\
=& 2^\ell \sum_{n, r\ge 1} r^{\ell-1} q^{2nr} - \sum_{n,r\ge 1} r^{\ell-1} q^{nr}
= 2^\ell \Phi_{\ell-1}(2\tau)  - \Phi_{\ell-1}(\tau).
\end{align*}
\end{proof}

\section{The Crank Statistic}\label{sec:crank}
In this section we will give proves of the main theorems of this paper. 
By the definition of $C(u;\tau)$ and the triple product formula
\begin{equation}\label{eqn:crank}
C(u;\tau) = -\frac{2\sin(\pi u) q^{\frac{1}{24}}\eta^2(\tau)}{\vartheta(u;\tau)}.
\end{equation}
The following lemma is useful.
\begin{lemma}\label{lem:basic} As formal power series, let 
$\sum_{r\ge 0} \frac{Z^r}{r!} c_r = \exp\( \sum_{\ell >0} \frac{Z^{\ell}}{\ell!} a_\ell \)$ then
$$c_r = \sum_{0\le s\le r} \frac{r!}{s!}  \sum_{\begin{subarray}{c} i_1+\cdots+i_s = r\\ i_j>0\end{subarray}} 
\frac{1}{i_1! i_2! \cdots i_s!} a_{i_1}a_{i_2} \cdots a_{i_s}.$$
\end{lemma}

\begin{proof}[Proof of Theorem \ref{thm:ag4.8}]
Applying Proposition \ref{prop:jacobi}  and Lemma \ref{lem:basic} to \eqref{eqn:crank} yields the result. 
\end{proof}


\begin{proof}[Proof of Theorem \ref{thm:twist}] We have the following Taylor expansion
$$
\cC(Z+\pi i; q) = \sum_{\ell\ge 0} \frac{Z^\ell}{\ell!} \sum_{n\ge 1} (\sum_{m} (-1)^m m^{\ell} M(m,n)) q^n$$
Note that by symmetry $\sum_{m} (-1)^m m^{\ell} M(m,n)$ is zero unless $\ell$ is even.   
Applying Proposition \ref{prop:jacobi}  and Lemma \ref{lem:basic}, we readily deduce the explicit formula for the 
twisted moment generating function. 

By 3.3.8 of \cite{EO} the algebra of quasimodular forms on $\Gamma_0(2)$ is generated by $E_2(\tau)$, $E_2(2\tau)$, 
and $E_4(2\tau)$.  Thus $(q)_\infty^{-1}(-q)_\infty^{-2} C_{2\ell}(-1,q)$ is a quasimodular form on $\Gamma_0(2)$. 
\end{proof}

\subsection{Crank Asymptotics}
In this subsection we use the method from \cite{BMR} to prove Theorem \ref{thm:crankAsymptotic}. We have 
$$
\cC(2\pi i u + \pi i;q ) = 
\frac{2\cos(\pi u) e^{\frac{i\pi \tau}{12} } \eta^2(\tau) }{ \vartheta(u + \frac{1}{2}; \tau)}.$$
To state the modular transformation properties of the twisted crank generating functions we 
denote the inverse of $a$ modulo $b$ by $[a]_b$.  When $2\mid b$ we may assume $[a]_b = [a]_{2b}$. We do 
this implicitly in what follows. 

\begin{proposition}\label{prop:crankMomentAsymp} For $2\mid k$ we have 
\begin{align*}
C_{2\ell}\(-1,e^{\frac{2\pi i}{k}(h+iz)}\) =& i^{\frac{3}{2}} e^{\frac{\pi i}{12 k} (h-[-h]_k)}
\chi^{-1}(h,[-h]_k,k) \cdot (-1)^{\frac{k}{2}} (-1)^{\frac{k+1-h[-h]_k}{2}} 
e^{2\pi i \frac{-h-\frac{k}{2}[-h]_k}{4}} \\
&\times 
e^{\frac{\pi}{12 k}(\frac{1}{z}-z)} \sum_{a+b+c = \ell} k^c \alpha(a,b,c) 
z^{-\frac{1}{2} - 2b-c} + O\(z^{-\frac{1}{2}-2\ell} e^{-\frac{\alpha}{k}\re{\frac{1}{z}}} \)
\end{align*}
for some $\alpha >0$ independent of $k$.
Furthermore, for $2\nmid k$ we have 
$$C_{2\ell}\(-1,e^{\frac{2\pi i }{k}(h+iz)}\) \ll z^{-2\ell -\frac{1}{2}} e^{-\frac{\pi}{6k} \re{\frac{1}{z} } }.$$
\end{proposition}

\begin{lemma}\label{lem:shift}
For appropriate $a$ and $b$ and $\ell\in \N$ we have 
$$\vt(a+\ell b; b) = (-1)^\ell e^{-\pi i \ell^2 b - 2\pi i \ell a} \vt(a;b).$$
\end{lemma}
\begin{proof}
The proof follows from induction on $\vt(a+b; b) = - e^{-\pi i b - 2\pi i a}\vt(a;b)$. 
\end{proof}

\begin{proof}[Proof of Proposition \ref{prop:crankMomentAsymp}]
Using the modular transformations for the Dedekind $\eta$-function and the 
Jacobi $\vt$-function, namely  
\begin{align}
\eta\(\frac{1}{k}(h+iz)\) =& \sqrt{\frac{i}{z}} \chi(h, [-h]_k,k) \eta\(\frac{1}{k}\([-h]_k + \frac{i}{z}\)\) \\
\vt\(u; \frac{1}{k}(h+iz)\) =& \sqrt{\frac{i}{z}} \chi^3(h,[-h]_k,k) e^{-\frac{\pi k u^2}{z}} \vt\(\frac{iu}{z}; \frac{1}{k}\(
[-h]_k +\frac{i}{z}\)\)
\end{align}
where $$\chi(h,[-h]_k,k) = i^{-\frac{1}{2}} \omega_{h,k}^{-1} e^{-\frac{\pi i}{12 k}([-h]_k -h)}$$
and $[-h]_k$ is the multiplicative inverse of $-h$ modulo $k$ and $\omega_{h,k}:= \exp(\pi i s(h,k))$.  Here, $s(h,k)$ is 
the usual Dedekind sum.  See \cite{knopp} and \cite{BMR} and the references therein for more on these 
transformation formulae. 
Hence
\begin{equation}\label{eqn:cranktrans} 
\cC(2\pi i u + \pi i ; e^{\frac{2\pi i}{k} (h+iz)}) = 2 \cos(\pi u) \chi^{-1}(h,[-h]_k, k)
\sqrt{\frac{i}{z}} \frac{\eta^2\(\frac{1}{k}([-h]_k + \frac{i}{z})\)^2 e^{\frac{\pi k (u+\frac{1}{2})^2}{z}}}{
\vt\(\frac{i(u+1/2)}{z} ; \frac{1}{k}([-h]_k + \frac{i}{z})\)} e^{\frac{\pi i}{12 k}(h+iz)}.
\end{equation}
Set $\tau =\frac{1}{k}\([-h]_k + \frac{i}{z}\)$ and write 
$$\vt\(\frac{i(u+1/2)}{z} ;\tau \) = 
\vt\(\frac{iu}{z} - \frac{[-h]_k}{2} + \frac{k}{2} \tau ; \tau\)$$

In the case $2\mid k$ set $\ell = \frac{k}{2}$. Applying Lemma \ref{lem:shift} we have 
$$\vt\(\frac{iu}{z}  - \frac{[-h]_k}{2} + \ell \tau ; \tau\) = (-1)^\ell e^{-\pi i \ell^2 \tau - 2\pi i \ell
\(\frac{iu}{z} - \frac{[-h]_k}{2}\)} \vt\( \frac{iu}{z} - \frac{[-h]_k}{2}; \tau\).$$
To calculate the asymptotic we use the triple product identity and abuse notation by setting 
$q = e^{2\pi i \tau}$ and $x = e^{-\frac{2\pi u}{z} - \pi i [-h]_k}$.
Then we have 
\begin{align*}
\vt\(\frac{iu}{z} - \frac{[-h]_k}{2} ; \tau\)^{-1} 
=& -\frac{q^{-\frac{1}{8}}}{2\sin\( \frac{iu\pi}{z} - \frac{\pi [-h]_k}{2}\)} \prod_{n\ge 1} (1-q^n)^{-1} 
(1-xq^n)^{-1} (1-x^{-1}q^n)^{-1} \\
=& -\frac{q^{-\frac{1}{8}}}{2 \(\frac{[-h]_k}{4}\) \cosh(\frac{u\pi}{z})} + \sum_{r\ge 0} a_r(z) \frac{u^r}{r!}
\end{align*}
with $a_r(z) \ll \abs{z}^{-r} e^{-\frac{7\pi}{4k} \re{\frac{1}{z}}}$ and $\(\frac{\cdot}{\cdot}\)$ is the Kronecker symbol.

Combining this with \eqref{eqn:cranktrans} we have 
\begin{align*}
\cC\(2\pi i u+\pi i; e^{\frac{2\pi i}{k}(h+iz)}\) =& 
-\frac{\cos(\pi u)}{\cosh\(\frac{\pi u}{z}\)} e^{-\frac{\pi i}{4k}([-h]_k + \frac{i}{z})} \sqrt{\frac{i}{z}}
e^{\frac{\pi i}{12k}(h+iz)} e^{\frac{\pi i}{6k}([-h]_k + \frac{i}{z})} \chi^{-1}(h,[-h]_k, k) \\ &\times  
e^{\frac{\pi k}{z}(u^2+u+\frac{1}{4})} (-1)^{\frac{k}{2}} e^{\frac{\pi i}{k}\(\frac{k}{2}\)^2 ([-h]_k + \frac{i}{z})
+ \pi i k (\frac{iu}{z} - \frac{[-h]_k}{2})} \(\frac{[-h]_k}{2}\)  +\sum_{r=0}^\infty a_r(z) \frac{u^r}{r!}
\\
=& -i^{\frac{3}{2}} e^{\frac{\pi i}{12 k}\( h-[-h]_k\)} \chi^{-1} (h,[-h]_k,k) 
(-1)^{\frac{k}{2}} i^{-1} \(\frac{-h}{4}\)e^{-\frac{\pi i}{4}k [-h]_k}
z^{-\frac{1}{2}} e^{\frac{\pi}{12 k}\(\frac{1}{z}-z\)} \\
&\times  \frac{\cos(\pi u) e^{\frac{\pi k u^2}{z}}}
{\cosh\(\frac{\pi u}{z}\)}
 +\sum_{r=0}^\infty a_r(z) \frac{u^r}{r!},
\end{align*}
where we have used the fact that $-h \equiv [-h]_k \pmod{4}$ with 
$a_r(z) \ll z^{-\frac{1}{2}-r} e^{-\frac{\alpha}{k} \re{\frac{1}{z}}}$ for some $\alpha >0$. For later reference, using  
$1-h[-h]_k \equiv 0 \pmod{2}$, we note that 
\begin{equation}
(-1)^{\frac{k}{2}} i^{-1} \(\frac{-h}{4}\)e^{-\frac{\pi i}{4}k [-h]_k} = (-1)^{\frac{k+1-h[-h]_k}{2}} 
e^{2\pi i \frac{-h-\frac{k}{2}[-h]_k}{4}}.
\end{equation}

Taylor expanding those functions depending on $u$, namely, 
$$\cos(\pi u) = \sum_{a\ge 0} \frac{(2\pi i u)^{2a}}{(2a)! 2^{2a}}\ \ \  \text{   and   } \ \ \ 
\frac{1}{\cosh\(\frac{\pi u}{z}\)} = \sum_{b\ge 0} E_{2b} \frac{(2\pi i u)^{2b} (-1)^b}{z^{2b} (2b)! 2^{2b}}$$
where $E_{2b}$ is the Euler number, we have 
$$\frac{\cos(\pi u)}{\cosh\(\frac{\pi u}{z}\)} e^{\frac{\pi k u^2}{z}} 
 = \sum_{\ell \ge 0} \frac{(2\pi i u)^{2\ell}}{(2\ell)!} 
\sum_{a + b +c= \ell} k^c \alpha(a,b,c)z^{-2b-c}.$$
Thus, we finish the proof in the case $2\mid k$.

In the case when $2\nmid k$ we write $k = 2\ell+1$ with $\ell\ge 0$.
Then $$\vt\(\frac{iu}{z} + \frac{i}{2z}; \tau\) = (-1)^\ell e^{-\pi i (\ell^2+\ell)\tau - 2\pi i \ell (\frac{iu}{z}
-\frac{[-h]_k}{2})} \vt\(\frac{iu}{z} - \frac{[-h]_k}{2} + \frac{\tau}{2} ; \tau\).$$
Again, with $x = e^{2\pi i \(\frac{iu}{z} - \frac{[-h]_k}{2} \)}$ and $q = e^{2\pi i \tau}$,
\begin{align*}
\vt\(\frac{iu}{z} +\frac{i}{2z} ; \tau\)^{-1}=& (-1)^{\ell+1}
q^{\frac{\ell^2 + \ell}{2} + \frac{1}{8}} x^{\frac{2\ell+1}{2}} \prod_{n=1}^\infty (1-q^n)^{-1} (1-xq^{n-\frac{1}{2}})^{-1}
(1-x^{-1}q^{n-\frac{1}{2}})^{-1}.
 \end{align*} 
 Combining this with \eqref{eqn:cranktrans} and 
Taylor expanding in $u$, as above, we obtain the result.
\end{proof}

We follow \cite{BMR} closely and give the following general result 
which yields the proof of Theorem \ref{thm:crankAsymptotic}.
Assume that
\begin{equation*}
F_{r,\ell}\(e^{ \frac{2 \pi i}{k}(h+iz) }\) = \sum_n c_{r,\ell}(n)e^{\frac{2
\pi in}{k}(h+iz) }
\end{equation*}
is a holomorphic function of $z$  satisfying
\begin{align*}\label{eqn:Fasymptotic}
F_{r,\ell}\(e^{\frac{2\pi i}{k}(h+iz)}\) =&
 -i^{\frac{3}{2}} e^{\frac{\pi i}{12 k} (h-[-h]_k)}
\chi^{-1}(h,[-h]_k,k) \cdot (-1)^{\frac{k+1-h[-h]_k}{2}} 
e^{2\pi i \frac{-h-\frac{k}{2}[-h]_k}{4}} \\
&\times 
e^{\frac{\pi}{12 k}(\frac{1}{z}-z)} \sum_{a+b+c = \ell} (rk)^c \alpha(a,b,c) 
z^{-\frac{1}{2} - 2b-c} + E_{r,\ell, k}(z)(z) 
\end{align*}
and for $2\mid k$ 
\begin{equation}\label{eqn:Fasymptotic2}
F_{r,\ell}\(e^{\frac{2\pi i}{k}(h+iz)}\)  \ll E_{r,\ell, k}(z)
\end{equation}
with $E_{r,\ell,k}(z) \ll_{r,\ell} z^{-\frac{1}{2}-2\ell} e^{-\frac{\beta}{k} \re{\frac{1}{z}}}$
for some $\beta>0$ 
with 
$\alpha(a,b,c)$  defined as in \eqref{eqn:alpha}.

\begin{theorem}\label{thm:circlemethod}
With $F_{r,\ell}$ and $c_{r,\ell}$ as above
we have
\begin{align*}
c_{r,\ell}(n)=&\pi \sum_{1\le k \leq
\sqrt{n}/2 }\frac{(-1)^{k+\lfloor \frac{k+1}{2}\rfloor} A_{2k}\(
n-\frac{k(1+(-1)^k)}{4}\) }{k}\sum_{a+b+c=\ell}(2kr)^c
\alpha(a, b, c) \\ & \hspace{1in} \times (24n-1)^{b+\frac{c}{2}-\frac{1}{4}}
I_{\frac{1}{2}-2b-c}\left(\frac{\pi
\sqrt{24n-1}}{12k}\right) +O\Big(n^{2\ell+ \epsilon}\Big),
\end{align*}
where $A_k(n)$ is the Kloosterman sum defined in \eqref{Kloosterman} and
$\epsilon>0$. 
\end{theorem}

The proof is nearly identical to that of Theorem 4.1 of \cite{BMR}, so we do not include it here. 
One slight difference is that we must appeal to  
\begin{equation}\label{eqn:multiplierIdentity}
(-1)^{c+\lfloor \frac{c+1}{2}\rfloor} A_{2c}\(
n-\frac{c(1+(-1)^c)}{4}\) = \sum_{h\pmod{2c}} \omega_{h,2c}(-1)^{\frac{2c+1-h[-h]_{2c} }{2}} e\(\frac{-h-c[-h]_{2c} }{4}
-\frac{nh}{2c}\)
\end{equation}
where
$$\omega_{h,k} = i^{-\frac{1}{2}} \chi(h,[-h]_k,k)^{-1} e^{-\frac{\pi i}{12k}([-h]_k - h)}.$$ 
Equation \eqref{eqn:multiplierIdentity} follows from 
$$(-1)^{\lfloor \frac{c+1}{2}\rfloor} A_{2c}\(
n-\frac{c(1+(-1)^c)}{4}\) = \sum_{d\pmod{2c}} \omega_{-d,2c}(-1)^{\frac{2c+1+ad}{2}} e\(\frac{a-3dc}{4}+\frac{nd}{2c}\),$$
where $ad\equiv 1 \pmod{2c}$ 
which can be found in \cite{BO}.

\section{The Rank Statistic}\label{sec:rank}
A second partition statistic is the rank statistic of Dyson \cite{Dy44}.  The rank statistic is of great interest, but is not covered by the results presented here.   
Zwegers \cite{Zwegers} proved that  the
two variable generating 
function,  $R(x;q)$, for the rank is a mock Jacobi form. 
For example,  
the mock theta function $f(q) = R(-1;q)$  was introduced by Ramanujan in his final letter to Hardy. 
This function and the other mock theta functions have a long history (see Zagier's Bourbaki lecture \cite{Zag09}). 
Recently the work of Zwegers was used by Bringmann and Ono \cite{BO, BO2} to provide a thorough study of this 
function.

The Taylor coefficients with respect to the $x$ variable of $R(x;q)$ are not quasimodular forms, 
but instead are quasimock modular forms \cite{BGM}.  
As a result, the automorphic properties of the rank moment generating functions are more complicated 
than those for the crank.
An extensive study, relying on the so-called ``Rank-Crank PDE'' of Atkin and Garvan
\cite{AG03}, was completed by Bringmann, Garvan and Mahlburg
\cite{BGM}.  However, the perspective of Taylor coefficients of Jacobi forms 
simplifies some elements of that study and 
makes further study accessible. 
For instance, it would be of interest to understand these forms explicitly.

The Fourier coefficients of $f(q) = R(-1;q)$ should be compared with the 
the twisted crank discussed in this note. The perspective of Taylor coefficients 
makes a full comparison of the twisted crank and rank
moments, along the lines of that given in \cite{BMR} possible. It would be of additional interest 
to have a  combinatorial or $q$-series understanding of these twisted moments, similar to that given by
Andrews \cite{And07} and Garvan \cite{Garvan}, for the rank and crank moments.

\section*{Acknowledgements}
This work was completed during a visit to Mathematical Sciences Research Institute in Berkeley.  The author thanks
MSRI for a hospitable work environment.   The author thanks Karl Mahlburg for helpful discussions.  Finally the author
thanks Jeff Lagarias for many comments and discussions about the paper.

\end{document}